\documentclass[a4paper,reqno,11pt]{amsart}

\usepackage{amsmath,amsthm,amssymb,mathtools,url,breqn}

\usepackage{amssymb}
\usepackage{amstext}
\usepackage{amsmath}
\usepackage{amscd}
\usepackage{amsthm}
\usepackage{amsfonts}
\usepackage{enumerate}
\usepackage{graphicx}
\usepackage{color}
\usepackage{here} 
\usepackage{bm} 
\usepackage{mathrsfs}
\usepackage{mathtools}
\usepackage{latexsym}
\usepackage{booktabs}
\usepackage{fancyhdr}
\usepackage{subcaption}
\usepackage{tikz}
\usepackage{tikz-cd}
\usetikzlibrary{arrows}
\usetikzlibrary{decorations.markings}
\usetikzlibrary{patterns,decorations.pathreplacing}

\newcommand{\Hom}{\textrm{Hom}}

\newcommand{\pd}{\textrm{pd}}

\newcommand{\depth}{\textrm{depth}}

\newcommand{\link}{\mathrm{link}}

\newcommand{\tr}{\textrm{tr}}
\newcommand{\bfa}{\mathbf{a}}
\newcommand{\bfb}{\mathbf{b}}
\newcommand{\bfv}{\mathbf{v}}
\newcommand{\bfu}{\mathbf{u}}
\newcommand{\bfe}{\mathbf{e}}
\newcommand{\bfx}{\mathbf{x}}

\newcommand{\ZZ}{\mathbb{Z}}
\newcommand{\NN}{\mathbb{N}}
\newcommand{\QQ}{\mathbb{Q}}
\newcommand{\RR}{\mathbb{R}}

\newcommand{\kk}{\Bbbk}

\def\opn#1#2{\def#1{\operatorname{#2}}} 
\opn\chara{char} \opn\length{\ell} \opn\pd{pd} \opn\rk{rk}
\opn\projdim{proj\,dim} \opn\injdim{inj\,dim} \opn\rank{rank}
\opn\depth{depth} \opn\grade{grade} \opn\height{height}
\opn\embdim{emb\,dim} \opn\codim{codim}

\opn\Tr{Tr} \opn\bigrank{big\,rank}
\opn\superheight{superheight}\opn\lcm{lcm}
\opn\trdeg{tr\,deg}
\opn\reg{reg} \opn\lreg{lreg} \opn\ini{in} \opn\lpd{lpd}
\opn\size{size}\opn{\mult}{mult}
\opn\aff{aff} \opn\con{conv} \opn\relint{relint} \opn\cok{coker}
\opn\img{Im} \opn\cn{cn} \opn\inte{int} \opn\vol{vol}
\opn\link{link} \opn\star{star}

\makeatletter
\@namedef{subjclassname@2020}{%
  \textup{2020} Mathematics Subject Classification}
\makeatother

\theoremstyle{plain}
\newtheorem{Theorem}{Theorem}[section]
\newtheorem{Lemma}[Theorem]{Lemma}
\newtheorem{Corollary}[Theorem]{Corollary}
\newtheorem{Thm}[Theorem]{Theorem}
\newtheorem{Proposition}[Theorem]{Proposition}

\theoremstyle{definition}
\newtheorem{Definition}[Theorem]{Definition}
\newtheorem{Example}[Theorem]{Example}
\newtheorem{Examples}[Theorem]{Examples}
\newtheorem{Remark}[Theorem]{Remark}

\begin{document}

\title{Comparing generalized Gorenstein properties in semi-standard graded rings}
\author[S. Miyashita]{Sora Miyashita}
\address[S. Miyashita]{Department of Pure And Applied Mathematics, Graduate School Of Information Science And Technology, Osaka University, Suita, Osaka 565-0871, Japan}
\email{u804642k@ecs.osaka-u.ac.jp}
\maketitle

\begin{abstract}
Semi-standard graded rings are a generalized notion of standard graded rings. In this paper, we compare generalized notions of the Gorenstein property in semi-standard graded rings. We discuss the commonalities between standard graded rings and semi-standard graded rings, as well as elucidate distinctive phenomena present in semi-standard graded rings that are absent in standard graded rings.
\end{abstract}

\section{Introduction}
With the development of non-Gorenstein Cohen-Macaulay analysis, various generalized properties of Gorenstein rings have been defined. Notable examples include nearly Gorenstein, almost Gorenstein, and level property.
In particular, comparisons of these properties have been done in \cite{herzog2019trace, higashitani2022levelness,miyashita2022levelness,moscariello2021nearly}.
Nearly Gorenstein and almost Gorenstein rings have been studied in various classes such as standard graded rings \cite{higashitani2016almost,miyashita2022levelness}, numerical semigroup rings \cite{kumashiro2023nearly,moscariello2021nearly}, and Ehrhart rings \cite{hall2023nearly,miyazaki2022gorenstein}.
In particular, Higashitani \cite{higashitani2016almost} developed the theory of almost Gorenstein standard graded domains in terms of their $h$-vector.
It is known that the almost Gorenstein property in standard graded domains can be determined by using their $h$-vector and Cohen-Macaulay type (see \cite[Corollary 2.7]{higashitani2016almost}).
Moreover, the author proved that in standard graded affine semigroup rings, the last term of its $h$-vector of non-Gorenstein nearly Gorenstein rings is always at least 2 (see \cite[Theorem 4.4]{miyashita2022levelness}). This is a useful tool when we compare nearly Gorenstein property with other properties in standard graded affine semigroup rings.
Furthermore, the level ring, a prominent generalization of Gorenstein rings, is a notion defined for graded rings.
Its behavior of standard graded rings has been extensively studied (see \cite{geramita2007hilbert,hibi1988level,yanagawa1995castelnuovo}). On the other hand, for the generalized class of standard graded rings, known as semi-standard graded rings, various studies have been conducted regarding level property, especially in relation to their $h$-vector (see \cite{haase2020levelness,higashitani2018non,stanley1991hilbert}).
Let us recall the definitions of standard graded and semi-standard graded.
\begin{Definition}
Let $R=\bigoplus_{i \in \NN} R_i$ be a graded Noetherian $\kk$-algebra over a field $R_0=\kk$.
If $R =\kk[R_1]$, that is, $R$ is generated by $R_1$ as a $\kk$-algebra, then we say $R$ is {\it standard graded}.  
If $R$ is finitely generated as a $\kk[R_1]$-module, then we say $R$ is {\it semi-standard graded}.
\end{Definition}
The {\it Ehrhart rings}  of lattice polytopes and the face rings of simplicial posets (see \cite{stanley1991f}) are typical classes of semi-standard graded rings.
From perspective of combinatorial commutative algebra, the concept of semi-standard graded rings naturally arises in this context.

If $R$ is a semi-standard graded ring of dimension $d$, its Hilbert series is of the form
$$\sum_{i \in \NN} (\dim_\kk R_i) t^i =\frac{h_0+h_1 t + \cdots +h_st^s}{(1-t)^d}$$ 
for some integers $h_0, h_1, \cdots, h_s$ with $\sum_{i=0}^s h_i \ne 0$ and $h_s \ne 0$.
We call the integer sequence $(h_0,h_1,\cdots,h_s)$ the {\it $h$-vector} of $R$ and denote it as $h(R)$.
Moreover, we call $s$ the {\it socle degree} of $R$ and denote it as $s(R)$.
We always have $h_0=1$.
If a semi-standard graded ring $R$ is Cohen--Macaulay,  its $h$-vector satisfies
$h_i \geqq 0$ for every $i$. Further, if $R$ is standard graded, we have $h_i > 0$ for every $i$. 
For further information on the $h$-vectors of Cohen--Macaulay semi-standard (resp. standard) graded rings, see \cite{stanley1991hilbert} (resp. \cite{stanley2007combinatorics}).

Goto, Takahashi, and Taniguchi \cite{goto2015almost} compare almost Gorenstein and level properties in standard graded rings, while the author \cite{miyashita2022levelness} compares nearly Gorenstein and level properties in standard graded rings. These studies have gradually revealed the compatibility of these properties in standard graded rings.
However, there are still many unknown things in the context of semi-standard graded rings. Therefore, the main theme of this paper is to compare nearly Gorenstein, almost Gorenstein, and level properties in semi-standard graded rings.

In this paper, we apply the techniques developed for standard graded rings, as seen in \cite{higashitani2016almost,miyashita2022levelness}, to the case of semi-standard graded rings. We establish the theory for semi-standard graded rings and extend several well-known results about standard graded rings to the case of semi-standard graded rings.
For instance, in Corollary \ref{NGhvector},
we extend the statement regarding the last term of the $h$-vector of standard graded rings to the semi-standard case.
We extend this further and demonstrate the following statement regarding trace ideals of semi-standard graded rings. This is applicable even in the non-Cohen-Macaulay case.

\newtheorem*{MainTheorem}{\rm\bf Theorem~\ref{TraceIdealAffine}}
\begin{MainTheorem}
{\it Let $R=\kk[S]$ be a
semi-standard graded affine semigroup ring.
Let $I$ be a non-principle ideal of $R$
and let $b=\min \{i: I_i\neq 0\}$.
If $\depth R \geqq 2$ and
$$(\bfx^\bfe : \bfe \in E_S)R \subseteq \tr(I),$$
then we have $\dim_\kk I_b \geqq 2$.
}
\end{MainTheorem}

\newtheorem*{ABC}{\rm\bf Corollary~\ref{NGhvector}}
\begin{ABC}
{\it Let $R=\kk[S]$ be a semi-standard graded Cohen-Macaulay affine semigroup ring.
  If $R$ is not Gorenstein and $(\bfx^\bfe : \bfe \in E_S)R \subseteq \tr(\omega_{R})$,
  then $h_{s(R)} \geqq 2$.
In particular, if $R$ is non-Gorenstein nearly Gorenstein, then $h_{s(R)} \geqq 2$.}
\end{ABC}

The following theorem regarding the necessary and sufficient condition to be level and almost Gorenstein was also known in the case of standard graded rings, as shown in \cite{goto2015almost}. We give a new proof by using Stanley's inequalities (see Corollary \ref{prop1}), and extend their results to the case of semi-standard graded rings.

\newtheorem*{MbinTheorem}{\rm\bf Theorem~\ref{AGandLevel}}
\begin{MbinTheorem}
{\it Let $R$ be a Cohen--Macaulay semi-standard graded ring with $\dim R>0$.
Suppose that $R$ is not Gorenstein.
Then the following conditions are equivalent:}
\begin{itemize}
\item[(1)] {\it$R$ is almost Gorenstein and level;}
\item[(2)] {\it$R$ is generically Gorenstein and $s(R)=1$.}
\end{itemize}
\end{MbinTheorem}

At the same time, we investigate specific properties that do not hold in standard graded rings but hold for semi-standard graded rings.
For instance, there are intriguing differences
if we consider nearly Gorenstein affine semigroup rings with projective dimension 2. In the case of standard graded affine semigroup rings, there exist non-Gorenstein yet nearly Gorenstein rings with projective dimension 2, and their characterization is provided for the case of projective monomial curves (see \cite[Theorem A]{miyashita2023nearly}). However, in the case of non-standard semi-standard graded rings, it turns out that there is no non-Gorenstein nearly Gorenstein ring with projective dimension 2.
\newtheorem*{MMbinTheorem}{\rm\bf Theorem~\ref{pd2NGsemi}}
\begin{MMbinTheorem}
    {\it Let $R$ be a non-standard semi-standard graded Cohen-Macaulay affine semigroup ring with projective dimension 2.
    Then the following conditions are equivalent:}
\begin{itemize}
\item[(1)]{\it$R$ is nearly Gorenstein;}
\item[(2)]{\it$R$ is Gorenstein.}
\end{itemize}
\end{MMbinTheorem}

Moreover, in standard graded affine semigroup rings, it is known that there is no instance that it is non-level almost Gorenstein and nearly Gorenstein(see Theorem \ref{AGandLevel} and Theorem \ref{AGandNG!}). However, in the case of semi-standard graded affine semigroup rings, such a special family is known to exist when the socle degree is 2. For the case of dimension 2, this family can be characterized as follows. The proof of this assertion relies significantly on the proof presented in \cite[Theorem 3.5]{higashitani2018non}.

\newtheorem*{MdinTheorem}{\rm\bf Theorem~\ref{nonlevelAGwithSocleDeg2}}
\begin{MdinTheorem}
{\it Let $R=\kk[S]$ be a Cohen--Macaulay semi-standard graded affine semigroup ring with $\dim R=s(R)=2$. Then the following conditions are equivalent:}
\begin{itemize}
\item[(1)]{\it$R$ is non-level and almost Gorenstein;}
\item[(2)]{\it$S \cong \langle \{(2i,2n-2i): 0 \leqq i \leqq n \} \bigcup \{ (2j+2k-1,4n-2j-2k+1): 0 \leqq j \leqq n-1 \} \rangle$ for some $n \geqq 2$ and $1 \leqq k \leqq n+1$.}
\end{itemize}
{\it Moreover, if this is the case, then $R$ is always nearly Gorenstein and $h(R)=(1,n-1,n)$.}
\end{MdinTheorem}

Furthermore, it is known that every $1$-dimensional almost Gorenstein ring is nearly Gorenstein ring (see \cite[Proposition 6.1]{herzog2019trace}). If we consider non-standard semi-standard graded affine semigroup rings with socle degree 2, we can establish that a 2-dimensional version of this statement holds true.

\newtheorem*{McinTheorem}{\rm\bf Theorem~\ref{AGisNGw}}
\begin{McinTheorem}
{\it Let $R$ be a non-standard semi-standard graded Cohen-Macaulay affine semigroup ring with $\dim R=s(R)=2$.
If $R$ is almost Gorenstein, then it is nearly Gorenstein.}
\end{McinTheorem}

The structure of this paper is as follows.
In Section 2, we prepare some definitions and facts for the discussions later.
In Section 3, we first organize the general theory concerning nearly Gorenstein semi-standard graded affine semigroup rings based on \cite{miyashita2022levelness}. We extend the results for nearly Gorenstein standard graded affine semigroup rings from \cite{miyashita2022levelness} to the case of semi-standard graded.
Moreover, we prove a statement concerning the trace ideal of (not necessarily Cohen-Macaulay) semi-standard graded affine semigroup rings with depth greater than or equal to 2.
Additionally, we discuss special properties of nearly Gorenstein semi-standard graded rings with small projective dimensions and provide a special family of nearly Gorenstein rings with socle degree 3.
In Section 4, we establish the general theory for almost Gorenstein semi-standard graded rings based on \cite{higashitani2016almost}. We demonstrate that the theory regarding $h$-vectors can be directly applied to the case of semi-standard graded rings.
In Section 5, we compare between almost Gorenstein rings and level rings. Utilizing the framework organized in Section 4, and considering cases with small socle degrees, we explore relationships between almost Gorenstein and level properties by using Stanley's inequalities.
In Section 6, we compare almost Gorenstein rings with nearly Gorenstein rings. For semi-standard graded affine semigroup rings with both socle degree and dimension equal to 2, we completely determine the structure of non-level and almost Gorenstein rings. Furthermore, we prove that this structure coincides with the special family of nearly Gorenstein rings introduced in Section 3.

\subsection*{Acknowledgement}
I am grateful to professor
Akihiro Higashitani for his
comments and instructive discussions.
Thanks to his very helpful advice, I could discover Theorem \ref{TraceIdealAffine}.

\section{Preliminaries}
Let $\kk$ be a field,
and let $R$ be an $\mathbb{N}$-graded $\kk$-algebra
with a unique graded maximal ideal $\mathbf{m}$.
Apart from Section 2, we always assume that $R$ is Cohen-Macaulay and admits a canonical module $\omega_R$.
\begin{itemize}
\item Let $\omega_R$ denote a canonical module of $R$. Let $a(R)$ denote the $a$-invariant of $R$, i.e., $a(R)=-\min\{n:(\omega_R)_n \neq 0\}$.
  \item Let $r(R)$ be the Cohen-Macaulay type of $R$,
  and let $\pd(R)$ be the projective dimension of $R$.
\item For a graded $R$-module $M$ and $N$, we use the following notation:
\begin{itemize}
\item Let $\mu(M)$ denote the number of minimal generators of $M$ as an $R$-module.
\item Let $e(M)$ denote the multiplicity of $M$. Then the inequality $\mu(M) \leq e(M)$ always holds. 
\item Fix an integer $k$.
Let $M(-k)$ denote the $R$-module whose grading is given by $M(-k)_n=M_{n-k}$
for any $n \in \mathbb{Z}$.
Moreover, if $k>0$, we write $M^{\oplus k} = M \oplus M \oplus \cdots \oplus M$ ($k$ times).
\item Let $\tr_R(M)$ be the sum of ideals $\phi(M)$
with $\phi \in \Hom^*_R(M,R)$. That is,
\[\tr_R(M)=\sum_{\phi \in \Hom^*_R(M,R)}\phi(M)\]
where $\Hom^*_R(M,R)=\{ \phi \in \Hom_R(M,R): \phi \text{\;is\;a\;graded\;homomorphism} \}.$
When there is no risk of
confusion about the ring we simply write $\tr(M)$.

If $M$ and $N$ are isomorphic as graded $R$-module,
then $\tr(M)=\tr(N)$.
\end{itemize}
\end{itemize}

Let us recall the definitions and facts of the nearly Gorenstein property and level property
of graded rings.

\begin{Definition}[{see \cite[Chapter III, Proposition 3.2]{stanley2007combinatorics}}]
We say that $R$ is\;$\textit{level}$\: if all the degrees of the minimal
generators of $\omega_R$ are the same.
\end{Definition}
\begin{Definition}[{see \cite[Definition 2.2]{herzog2019trace}}]
We say that $R$ is $\textit{nearly Gorenstein}$ if $\tr(\omega_R) \supseteq \mathbf{m}$.
In particular,  $R$ is nearly Gorenstein but not Gorenstein
if and only if $\tr(\omega_R) = \mathbf{m}$.
\end{Definition}

Let $R$ be a ring and $I$ an ideal of $R$ containing
a non-zero divisor of $R$. Let $Q(R)$ be the total
quotient ring of fractions of $R$ and set
$\displaystyle I^{-1}:=\{x \in Q(R):xI\subset R\}.$ Then
\begin{equation}\numberwithin{equation}{section}\label{traceformula}
\displaystyle \tr(I)=II^{-1}
\end{equation}
(see \cite[Lemma 1.1]{herzog2019trace}).

If $R$ is an $\mathbb{N}$-graded ring, then
$\omega_R$ is isomorphic to an ideal $I_R$ of $R$ as an $\mathbb{N}$-graded module
up to degree shift if and only if $R_{\mathfrak{p}}$ is Gorenstein
for every minimal prime ideal $\mathfrak{p}$
(for example, if $R$ is a domain).
We call $I_R$ the canonical ideal of $R$.

\begin{Remark}[{see \cite[Chapter I, Section 12]{stanley2007combinatorics}}]
Fix an integer $n$ with $n \geqq  2$. If $R$ is an $\mathbb{N}^n$-graded domain, then
$\omega_R$ is isomorphic to an ideal of $R$ as an $\mathbb{N}^n$-graded module
up to degree shift.
\end{Remark}

We recall some definitions about affine semigroups.
An \textit{affine semigroup} $S$ is a finitely generated sub-semigroup of $\mathbb{Z}^d.$
For $X \subseteq S$, we denote
by $\langle X \rangle$
the smallest sub-semigroup of $S$ containing $X$.
We denote the group generated by $S$ by $\ZZ S$,
the convex cone generated by $S$ by $\RR_{\geqq 0} S\subseteq \RR^d$
and the normalization by $\overline{S}=\ZZ{S} \cap \RR_{\geqq 0} S$.
The affine semigroup $S$ is pointed
if $S\cap(-S)=\{0\}$.
We can check easily that every semi-standard graded affine semigroup is pointed.
It is known that $S$ is pointed if and only if
the associated cone $C=\RR_{\geqq0} S$ is pointed ({see \cite[Lemma 7.12]{miller2005combinatorial}}).
Moreover, every pointed affine semigroup $S$ has
a unique finite minimal generating set ({see \cite[Proposition 7.15]{miller2005combinatorial}}).
Thus $C=\RR_{\geqq0} S$ is a finitely generated cone.
A face $F \subseteq S$ of $S$ is a subset such that
for every $a,b \in S$ the following holds:
\[a+b \in F \Leftrightarrow a \in F \text{\;and\;} b \in F.\]
The dimension of the face $F$ equals the rank of $\ZZ F$.
The $1$-dimensional faces of a pointed semigroup $S$ are called its extremal rays.
We prepare the following basic lemma for proving Theorem 4.4.
We denote $(\bfa,\bfb)$ as inner product of $\bfa, \bfb \in \RR^d$.
\begin{Lemma}[{see \cite[Lemma 2.6]{miyashita2022levelness}}]\label{extremalpoly}
Let $d \geqq 2$ and let $S$ be a $d$-dimensional
pointed affine semigroup, and let $C=\RR_{\geqq 0} S$.
Let $E$ be the set of extremal rays of $C$.
If $\bfx \notin C$,
then there exists $l \in E$
such that $(\bfx + l) \cap C = \emptyset$.
\end{Lemma}

\begin{Thm}[{see \cite[Theorem 3.1]{katthan2015non}}]\label{holedecomp}
  Let $S$ be a pointed affine semigroup.
  There exists a (not necessarily disjoint) decomposition
  \begin{equation}\label{hole}\overline{S} \setminus S = \bigcup_{i=1}^l(s_i+\ZZ F_i) \cap \RR_{\geqq 0}S
  \end{equation}
  with $s_i \in \overline{S}$
  and faces $F_i$ of $S$.
  \end{Thm}

  A set
  $s_i+\ZZ F_i$ from (\ref{hole}) is called
  a $j$-dimensional family of holes,
  where $j$ is the dimension of $F_i$.
  
  \begin{Thm}[{see \cite[Theorem 5.2]{katthan2015non}}]\label{S2}
  Let $S$ be a pointed affine semigroup of dimension $d \geqq 2$.
Then $\depth \kk[S] \geqq 2$
if and only if
every family of holes has dimension at least $1$.
    \end{Thm}

\begin{Thm}[{see \cite[Corollary 3.2 and Corollary 3.5]{herzog2019trace}}]\label{trace}
  Let $S=\kk[x_1,\cdots,x_n]$ be a polynomial ring, let $\mathbf{n}=(x_1,\cdots,x_n)$ be the graded maximal ideal of $S$ and let
  \[\mathbb{F}:0 \rightarrow F_p \xrightarrow{\phi_p} F_{p-1} \rightarrow \cdots \rightarrow F_1 \rightarrow F_0 \rightarrow R \rightarrow 0\]
  be a graded minimal free $S$-resolution of the Cohen-Macaulay ring $R=S/J$ with
  $J \subseteq \mathbf{n}^2$.
  Let $I_1(\phi_p)$ be the ideal of $R$ generated by all components of a
  representation matrix of $\phi_p$.
  Then the following hold.

  \textit{$(a)$}  Let $e_1,\cdots,e_t$ be a basis of $F_p$.
  Suppose that for $i=1,\cdots,s$
  the elements $\sum_{j=1}^t r_{ij}e_j$ generate
  the kernel of
  \[\psi_p : F_p \otimes R \longrightarrow F_{p-1} \otimes R,\]
  where
  \[\psi_p=\phi_p \otimes R.\]
  Then $\tr(\omega_R)$ is generated by the elements $r_{ij}$ with $i=1,\cdots,s$ and $j=1,\cdots,t$.

  \textit{$(b)$} If $r(R)=2$ and
  $R$ is a domain,
  then $\tr(\omega_R)=I_1(\phi_p)$.
  \end{Thm}

We state the necessary results about the minimal free resolution
of the codimension 2 lattice ideal based on \cite{peeva1998syzygies}.

\begin{Definition}
Let $S=\kk[x_1,\cdots,x_n]$ be a polynomial ring and let
$L$ be any sublattice of $\mathbb{Z}^n$.
We put
$\bfx^{\bfa} := {x_1}^{a_1}{x_2}^{a_2}\cdots {x_n}^{a_n}$
where $\bfa=(a_1,a_2, \cdots, a_n) \in \mathbb{N}^n.$
Then its associated \textit{lattice ideal} in $S$ is
\[I_L := (\bfx^\bfa-\bfx^\bfb : \bfa,\bfb \in \mathbb{N}^n \;\;\textit{and} \;\;\bfa-\bfb \in L ).\]
Prime lattice ideals are called \textit{toric ideals}.
Prime binomial ideals and toric ideals are identical ({see \cite[Theorem 7.4]{miller2005combinatorial}}).
\end{Definition}

\begin{Proposition}[{see \cite[Comments 5.9 (a) and Theorem 6.1 (ii)]{peeva1998syzygies}}]\label{codim2peeva}
Let $S=\kk[x_1,\cdots,x_n]$ be a polynomial ring.
If $I$ is a codimension $2$ lattice ideal of $S$
and the number of minimal generators of $I$
is 3, then
$R=S/I$ is Cohen-Macaulay and
the graded minimal free resolution of $R$ is the following form.
\[0\rightarrow S^{2} \xrightarrow {\left[
\begin{array}{cc}
u_1 & u_4 \\
u_2 & u_5 \\
u_3 & u_6
\end{array}
\right]
}
S^{3} \rightarrow
 S \rightarrow R \rightarrow 0,\]
where $u_i$ is a monomial of $S$ for all $1 \leqq i \leqq 6$.
\end{Proposition}
Note that a codimension 2 toric ideal $I$ is Cohen-Macaulay
but not Gorenstein if and only if the number of
minimal generators of $I$ is 3 ({see \cite[Remark 5.8 and Theorem 6.1]{peeva1998syzygies}}).

\section{Nearly Gorenstein property versus level property \\on semi-standard graded affine semigroup rings}
In this section, we will establish the theory of nearly Gorenstein affine semigroup rings and generalize the results of \cite[Section 4]{miyashita2022levelness} to the case of semi-standard graded affine semigroup rings.
Moreover, we prove a statement concerning the trace ideal of affine semigroup rings with depth greater than or equal to 2. This result holds even in the case of non-Cohen-Macaulay rings (Theorem 3.5).
Proposition \ref{extremal} is the key to extend  
the results of \cite[Section 4]{miyashita2022levelness}
to our case.

Let $S$ be a semi-standard graded affine semigroup,
and let $G_S=\{ \bfa_1,\cdots,\bfa_s \} \subseteq \mathbb{N}^d$ be the minimal generators of $S$.
Fix the affine semigroup ring $R=\kk[S]$.
For any $\bfa \in \mathbb{N}^d$, we set
\begin{equation} \notag
R_\bfa
=
\begin{cases}
\kk \bfx^\bfa & (\bfa \in S)\\
0 & (\bfa \notin S).
\end{cases}
\end{equation}
Then $R = \bigoplus_{\bfa \in \mathbb{N}^d }R_\bfa$ is a direct sum decomposition as an abelian group,
and $R_\bfa R_\bfb \subseteq R_{\bfa+\bfb}$ for any $\bfa,\bfb \in \mathbb{N}^d$.
Thus we can regard $R$ as an $\mathbb{N}^d$-graded ring.
For an $\NN^d$-graded ideal $I$ of $R$, we put
$$V_I=\left \{\bfv_1,\cdots,\bfv_r : \{\bfx^\bfv_1,\cdots,\bfx^\bfv_r\} \text{\;is the minimal generating system of\;} I \right \} \subseteq S.$$
If $R$ is Cohen-Macaulay, then $\omega_R$ is isomorphic to an ideal $I_R \subseteq R$ and denote $V_{\omega_R}$ as $V_{I_R}$.

Moreover, we use the following notations;

$({G_S})_{\min}
=
\{\bfv \in G_S
: \deg \bfx^\bfv \leqq
\deg \bfx^{\bfv_i} \;\text{for all}
\; 1 \leqq i \leqq r\}$,

$(V_I)_{\min}
=
\{\bfv \in V_I
: \deg \bfx^\bfv \leqq
\deg \bfx^{\bfv_i} \;\text{for all}
\; 1 \leqq i \leqq r\}$,

$S-V_I
= \left \{\bfu \in \ZZ S : \bfx^{\bfu} \in I^{-1} \right \}
= \{ \bfu \in \mathbb{Z}S : \bfu + \bfv \in S \; \text{for all}\; \bfv \in V_I\},$

$E_S=\{\bfa \in G_S : \NN \bfa \text{\;is a $1$-dimensional face of $S$}\}$.

Thus the following holds.

\begin{Proposition}\label{NGtokuchouAffine}
  Let $S$ be a pointed affine semigroup, let $R=\kk[S]$ and let $\bfa \in G_S$.
  The following are equivalent:
  
  \textit{$(a)$} $\bfx^\bfa \in \tr(I)$;

  \textit{$(b)$} There exist $\bfv \in V_I$
  and $\bfu \in S-V_I$ such that
  $\bfa = \bfu + \bfv$.
\begin{proof}
$(b) \Rightarrow (a)$ follows from equality (2.\ref{traceformula}).
We show $(a) \Rightarrow (b)$.
Assume that $\bfx^\bfa \in \tr(I)$.
Then we know from equality (2.\ref{traceformula})
that ${\bfx}^{\bfa_i} \in II^{-1}$.
Since $II^{-1}$ is an $\mathbb{N}^d$-graded homogeneous ideal and $R$ is a domain,
we can write $\bfx^{\bfa}=c_1\bfx^{\bfv_1}\bfx^{\bfu_1}+\cdots+c_r\bfx^{\bfv_r}\bfx^{\bfu_r}$,
where $c_i \in \kk$ and $\bfu_i \in \mathbb{Z}^d$.
Thus we have $\bfx^{\bfv_i+\bfu_i}=\bfx^{\bfa}$
for at least one $1 \leqq i \leqq r$,
as desired.
\end{proof}
\end{Proposition}

\begin{Corollary}\label{NGsemigroup}
  Let $S$ be a Cohen-Macaulay
  pointed affine semigroup.
  The following are equivalent:  
  \textit{$(a)$} $R=\kk[S]$ is nearly Gorenstein;

  \textit{$(b)$} For any $\bfa \in G_S$,
  there exist $\bfv \in V_{\omega_R}$
  and $\bfu \in S-V_{\omega_R}$ such that
  $\bfa = \bfu + \bfv$.
\end{Corollary}

\begin{Lemma}
  Let $S$ be a
  pointed affine semigroup and let $\bfa \in ({G_S})_{\min}$.
  If there exist $\bfv \in V_I$
  and $\bfu \in S-V_I$ such that
  $\bfa = \bfu + \bfv$,
  then $\bfv \in (V_I)_{\min}$.
\begin{proof}
If  $(V_I)_{\min}=V_I$, it is obvious.
Assume $(V_I)_{\min} \neq V_I$ and $\bfv \notin (V_I)_{\min}$.
Then there exists $\bfv' \in (V_I)_\min$
such that $\deg \bfx^{\bfv'-\bfv} <0$.
Since $\bfu \in S-V_I$,
Thus we get
$\bfa + \bfv'-\bfv=\bfu + \bfv' \in S \setminus \{0\}$.
Then $\deg \bfx^{\bfa + \bfv'-\bfv} \geqq \deg \bfx^{\bfa}$.
Therefore, 
we have $\deg \bfx^{\bfv'-\bfv}=\deg \bfx^{\bfa+\bfv'-\bfv}-\deg{ \bfx^{\bfa} } \geqq 0$,
which yields a contradiction.
\end{proof}
\end{Lemma}

\begin{Proposition}\label{extremal}
Let $S$ be a pointed affine semigroup.
If $R=\kk[S]$ is semi-standard graded, then $\{\bfx^\bfe : \bfe \in E_S\} \subseteq R_1$.
\begin{proof}
Take $\bfe \in E_S$.
Since $\bfx^\bfe \in R$ and $R$ is finitely generated $\kk[R_1]$-module,
we can write $\bfe=\bfa+\bfb$ for some $\bfa \in G_S$ and $\bfb \in S$ with $\deg \bfx^\bfa=1$.
Then $F=\NN \bfe$ is a face of $S$, so
$$\bfe=\bfa+\bfb \in F
\iff \bfa \in F \text{\;and\;} \bfb \in F.$$
Thus we get $\bfb=\mathbf{0}$ and $\bfx^\bfe=\bfx^\bfa \in R_1$, as desired.
\end{proof}
  \end{Proposition}

\begin{Thm}\label{TraceIdealAffine}
Let $R=\kk[S]$ be a
semi-standard graded affine semigroup ring.
Let $I$ be a non-principle ideal of $R$
and let $b=\min \{i: I_i\neq 0\}$.
If $\depth R \geqq 2$ and
$$(\bfx^\bfe : \bfe \in E_S)R \subseteq \tr(I),$$
then we have $\dim_\kk I_b \geqq 2$.
  \end{Thm}
\begin{proof}
Assume that $\depth R \geqq 2$, $(\bfx^\bfe : \bfe \in E_S)R \subseteq \tr(I)$ and $\dim_\kk I_b = 1$.
Then $(V_I)_\min=\{\bfv\}$ and we can take $\bfv' \in V_I$ such that $\deg \bfx^\bfv < \deg \bfx^{\bfv'}$.
Since $\bfv, \bfv' \in V_I$ and
$\bfv \neq \bfv'$,
we get $\bfv'-\bfv \in \ZZ{S}\setminus S$.
We claim
$\bfv'-\bfv \in \RR_{\geqq 0}S$.

Assume
$\bfv'-\bfv \notin \RR_{\geqq 0}S$.
Then by Lemma \ref{extremalpoly},
there exists $\bfa \in E_S$
such that
$\bfv' - \bfv + \bfa \notin S.$
On the other hand,
since $\bfa \in (G_S)_\min$ and $\bfx^\bfa \in \tr(I)$,
there exists $\bfu \in S-V_I$
such that $\bfa=\bfu+\bfv$ by Proposition \ref{NGtokuchouAffine} and Proposition \ref{extremal}.
Thus we get $\bfv' - \bfv + \bfe = \bfu+\bfv' \in S,$
this yields a contradiction.

Then, $\bfv'-\bfv \in \RR_{\geqq 0}S$ and
we get $\bfv'-\bfv \in \overline{S} \setminus S$.
Since $\bfv'-\bfv \in \overline{S}\setminus S$,
there exist
$s_i \in \overline{S}$ and
face $F$ of $S$ such that
$\bfv'-\bfv \in s_i +\ZZ F$
and $(s_i+\ZZ F) \cap S= \emptyset$
by using Theorem \ref{holedecomp}.
Since $\depth R \geqq 2$,
every family of holes has dimension at least $1$
by Theorem \ref{S2}.
So we have $\dim F \geqq 1$.
Since $\bfv'-\bfv \in s_i +\ZZ F$,
we can take $\bfx \in \ZZ F$ and write
$\bfv'-\bfv=s_i+\bfx$.
Thus, we get
$(\bfv'-\bfv + \ZZ F)\cap S=(s_i+\ZZ F) \cap S=\emptyset$.
In particular,
by taking an extremal ray $l$ of facet $F$,
we get
    ($\bfv'-\bfv + \ZZ l) \cap S = \emptyset$.
On the other hand,
we have $\bfv'-\bfv + l \in S$ because $\bfx^l \in \tr(\omega_R)$.
This yields a contradiction.
\end{proof}

\begin{Corollary}\label{NGhvector}
Let $R=\kk[S]$ be a semi-standard graded Cohen-Macaulay affine semigroup ring.
  If $R$ is not Gorenstein and $(\bfx^\bfe : \bfe \in E_S)R \subseteq \tr(\omega_{R})$,
  then $h_{s(R)} \geqq 2$.
In particular, if $R$ is non-Gorenstein nearly Gorenstein, then $h_{s(R)} \geqq 2$.
\end{Corollary}\begin{proof}
Recall that $\omega_R$ is isomorphic to an ideal $I_R$ of $R$ as graded $R$-module.
Since $\tr(\omega_R)=\tr(I_R)$
and $h_s= \dim_\kk (\omega_R)_{-a(R)}$,
the assertion follows from Theorem \ref{TraceIdealAffine}.
\end{proof}

\begin{Corollary}\label{type2NG}
For any semi-standard graded Cohen-Macaulay affine semigroup ring, if it is nearly Gorenstein with Cohen-Macaulay type 2, then it is level.
\end{Corollary}
\begin{proof}
If $R$ is not level, then $h_s=1$.
This contradicts Corollary \ref{NGhvector}.
\end{proof}

\begin{Examples}\label{exNG1}
By using Theorem \ref{trace} and $\mathtt{Macaulay2}$ $($\cite{M2}$)$, we can check the following:
\begin{itemize}
\item[(a)] Both of $T_1=\QQ[t,s^2t,s^4t,s^6t,s^8t,st^2,s^3t^2]$
and $T_2=\QQ[t,s^3t,s^{6}t,s^{9}t,st^2,s^4t^2]$ are
non-level nearly Gorenstein, and non-standard semi-standard graded affine semigroup ring with $r(T_i)=2i$, where $\deg s^at^b=b$ for any $a,b \in \NN$.
\item[(b)] The affine semigroup ring
$$R=\kk[x_1,x_2,x_3,x_4]/(x_2x_3^2-x_1x_4,x_2^3-x_3x_4,x_1x_2^2-x_3^3)$$
is nearly Gorenstein with $r(R)=\pd(R)=2$, where $\deg x_1=\deg x_2=\deg x_3=1$ and $\deg x_4 = 2$.
This is not semi-standard graded by Theorem \ref{pd2NGsemi}.
\end{itemize}
Example \ref{exNG1} (b) shows that nearly Gorenstein property does not imply $\dim_\kk (\omega_R)_{-a(R)} \geqq 2$ for non-semi-standard graded affine semigroup ring $S$
in general.
\end{Examples}

When the projective dimension is greater than or equal to 3, there are many examples of non-standard semi-standard graded nearly Gorenstein affine semigroup rings.
The following family consists of valuable examples of semi-standard graded affine semigroup rings that are non-level, almost Gorenstein, and nearly Gorenstein
(In the case of standard graded affine semigroup rings, such examples do not exist! Refer to Theorems \ref{s2ikahalevel} and \ref{AGandNG!}).
While we prove here that they are non-level and nearly Gorenstein, the proof of their almost Gorenstein property will be presented in the next section (see Theorem \ref{nonlevelAGwithSocleDeg2}).

\begin{Proposition}\label{NGfamily}
Fix $n \geqq 2$, $1 \leqq k \leqq n+1$ and define the affine semigroup $S$ as
$$S=\langle \{(2i,2n-2i): 0 \leqq i \leqq n \} \cup \{ (2j+2k-1,4n-2j-2k+1): 0 \leqq j \leqq n-1 \} \rangle.$$
Then $R=\kk[S]$ is nearly Gorenstein with $h(R)=(1,n-1,n)$ and
$r(R)=2n-1$.
\begin{proof}
Put $F_1 = \NN(2n,0)$,
$F_2 = \NN(0,2n)$ and put
$$C_i=\{w \in \ZZ S_\bfa \; : \;w+g \notin S_\bfa \; \textit{for any}\; g \in F_i\}$$
for $i=1,2$, respectively.
Denote by $\kk[\omega_S]$ the $R$-submodule of $R$
generated by $\{ \bfx^{v} : v \in \omega_{S} \}$,
where $\omega_S=-(C_1 \cap C_2)$.
By applying \cite[Theorem 3]{goto1976affine} to our case,
$\kk[\omega_S]$ is the canonical module of $R$
if $R$ is Cohen-Macaulay.
First we show $R$ is Cohen-Macaulay.
To check this, it is enough to check $\ZZ S= S \cup C_1 \cup C_2$ (see \cite[Theorem 1]{goto1976affine}).
We put
\begin{equation} \notag
X_1
=
\begin{cases}
\{(-1+2i,1-2i): k-n \leqq i \leqq 0\} & (1 \leqq k \leqq n)\\
\emptyset & (k=n+1),
\end{cases}
\end{equation}
\begin{equation} \notag
X_2
=
\begin{cases}
\{(-1+2i,1-2i) : 1 \leqq i \leqq k-1\} & (2 \leqq k \leqq n+1)\\
\emptyset & (k=1),
\end{cases}
\end{equation}
$Y_1=X_1 \cup \{(n,-n): n \in \ZZ_{>0} \}$ and
$Y_2=X_2 \cup \{(-n,n) : n \in \ZZ_{>0} \}$.
Then we have
$$C_1=\bigcup_{a\in Y_1} (a+\ZZ(2n,0)),
\;\;\;\;C_2=\bigcup_{a\in Y_2} (a+\ZZ(0,2n)).$$
Thus,
$\ZZ S=\{(i,2nj-i): i,j \in \ZZ\}=S \cup C_1 \cup C_2$
so $S$ is Cohen-Macaulay.
Moreover, we can calculate the canonical module as follows:
\begin{align}
\omega_S&=-(C_1 \cap C_2) \nonumber\\
&=\bigcup_{a\in X_1} (-a+\NN(2n,0))
\cup \bigcup_{a\in X_2} (-a+\NN(0,2n))
\cup
\{(i,2nj-i): \; 1 \leqq i \leqq 2n-1 \;\textit{and\;} j \geqq 1 \} \nonumber\\
&=\langle \{(1-2i,-1+2i) : k-n \leqq i \leqq k-1
\} \cup \{(2i,2n-i): 1 \leqq i \leqq n-1 \} \rangle
\nonumber.
\end{align}
Then $r(R)=2n-1$.
Moreover,
we can check
$R$ is nearly Gorenstein
by Proposition \ref{NGsemigroup}.
Lastly, since
$$R=R_1 \oplus R_1\langle x_1^{2j+2k-1}x_2^{4n-2j-2k+1}: 0 \leqq j \leqq n-1 \rangle,$$
we have $h(R)=(1,n-1,n)$.
\end{proof}
\end{Proposition}
The family of nearly Gorenstein rings $R$ for Proposition \ref{NGfamily} satisfies $\pd(R)=3$.
There is also another example of a nearly Gorenstein semi-standard graded affine semigroup ring $R$ with $\pd(R)=3$, such as the following.
\begin{Example}
$\QQ[u,s^2u,t^2u,s^2t^2u,su^2,s^3u^2]$
is nearly Gorenstein semi-standard graded affine semigroup ring, where $\deg s^at^bu^c=c$ for any $a,b,c \in \NN$.
We can check it in the same way as Examples \ref{exNG1}.
\end{Example}

A non-Gorenstein nearly Gorenstein standard graded affine semigroup ring with projective dimension 2 does exist, and its characterization is known in the context of projective monomial curves
(see \cite[Theorem A]{miyashita2023nearly}).
However, for non-standard semi-standard graded affine semigroup rings, there are no examples with projective dimension 2 that are nearly Gorenstein, except for those that are Gorenstein.

 \begin{Thm}\label{pd2NGsemi}
    Let $R$ be a non-standard semi-standard graded Cohen-Macaulay affine semigroup ring with $\pd(R) = 2$.
    Then the following conditions are equivalent:
\begin{itemize}
\item[(1)] $R$ is nearly Gorenstein;
\item[(2)] $R$ is Gorenstein.
\end{itemize}
\begin{proof}
It is enough to show that (1) implies (2).
Assume that $R$ is not Gorenstein.
We put $R=\kk[S]$ and $d=\dim R$.
Note that $n=d+2$ by the Auslander-Buchsbaum formula.
Since $R$ is a semi-standard graded affine semigroup ring, we may assume $d\geqq 2$.
By the assumption,
there exists a codimension 2 homogeneous prime binomial ideal $I$
such that
$I$ is minimally generated by three elements and
$R \cong A/I$, where $A=\kk[x_1,\cdots,x_n]$ is a polynomial ring.
Since $I$ is a codimension $2$ lattice ideal and the number of minimal generators of $I$
is 3,
the graded minimal free resolution of $R$ is of the following form by Proposition \ref{codim2peeva}.
Note that $R$ is level by Corollary \ref{type2NG} since $r(R)=2$.
\footnotesize
\begin{equation}\notag
0\rightarrow A(-\deg f_1-\deg u_1)^2  \xrightarrow {X =\left[
\begin{array}{cc}
u_1 & -u_4 \\
-u_2 & u_5 \\
u_3 & -u_6
\end{array}
\right]
}
A(-\deg f_1) \oplus A(-\deg f_2) \oplus A(-\deg f_3) \rightarrow
 A \rightarrow R \rightarrow 0. \small
\end{equation}
\normalsize
Here, $u_i$ is a monomial of $A$ for all $1 \leqq i \leqq 6$
and $f_1=u_1u_5-u_2u_4, f_2=u_3u_4-u_1u_6$ and $f_3=u_2u_6-u_3u_5$.
Since $I$ is a graded prime ideal, we have
$\gcd(u_1u_5,u_2u_4)=\gcd(u_3u_4,u_1u_6)=\gcd(u_2u_6,u_3u_5)=1$ and
\begin{equation}
\deg u_1u_5=\deg u_2u_4, \;\;\; \deg u_3u_4=\deg u_1u_6, \;\;\; \deg u_2u_6=\deg u_3u_5. \label{ddd}
\end{equation}

Moreover, since $R$ is level, we have
\begin{equation}
\deg u_1=\deg u_4, \;\;\;\deg u_2=\deg u_5, \;\;\;\deg u_3=\deg u_6. \label{ccc}
\end{equation}
\begin{itemize}
\item Let $d=2$.
Since $R$ is nearly Gorenstein,
$X$ may be assumed to have one of the following forms by Theorem \ref{trace}.

(i) $X =\left[
\begin{array}{cc}
x_1 & -x_4 \\
-x_2 & u_5 \\
x_3 & -u_6
\end{array}
\right]
$
or
(ii) $X =\left[
\begin{array}{cc}
x_1 & -x_3 \\
-u_2 & x_4 \\
x_2 & -u_6
\end{array}
\right]
$
or
(iii) $X =\left[
\begin{array}{cc}
x_1 & -x_3 \\
-x_2 & x_4 \\
u_3 & -u_6
\end{array}
\right]
$.

(For example,
there is also
a possibility that
$X =\left[
\begin{array}{cc}
u_1 & -x_2 \\
-u_2 & x_3 \\
x_1 & -x_4
\end{array}
\right]
$,
but this
can be regarded to be
the same as (i)).

\begin{itemize}
\item[(i)]
We can write
$X=\left[
\begin{array}{cc}
x_1 & -x_4 \\
-x_2 & x_1^ax_3^b \\
x_3 & -x_1^cx_2^d
\end{array}
\right]
$ for some $a,b,c,d \in \NN$ with $ac=0$.
Moreover, we have
$x_{i},x_{j} \in E_S$
for some
$1 \leqq i<j \leqq 4$
by Proposition \ref{extremal}.
\item[$\cdot$] If $(i,j)=(1,2)$, then we have
$(u_5,u_6)=(x_1,x_2^{\deg x_3})$ or $(x_3,x_1)$ or $(x_3,x_2)$.
$(u_5,u_6)=(x_1,x_2^{\deg x_3})$ implies $x_2^{\deg x_3+1}=x_1x_3$. Since $x_2 \in E_S$, this yields a contradiction.
$(u_5,u_6)=(x_3,x_1)$ or $(x_3,x_2)$ implies $\deg x_i=1$ for all $1 \leqq i \leqq 4$, this leads to a contradiction since $R$ is standard graded.
\item[$\cdot$] If $(i,j)=(1,3)$, then we have
$(u_5,u_6)=(x_3^{\deg x_2},x_1)$ or $(x_3,x_2)$.
By the same argument as above, this yields a contradiction.
\item[$\cdot$] If $(i,j)=(1,4)$, according to (\ref{ddd}), we obtain
$\deg x_2=a+b \deg x_3$ and
$\deg x_3=c+d \deg x_2$,
so we have $(bd-1)\deg x_3=-(c+ad) \leqq 0.$
Thus $bd=0$ or $1$.
If $bd=0$, then we get $x_1^{a+1}=x_2x_4$ or $x_1^{c+1}=x_3x_4$. This leads to a contradiction since $x_1 \in E_S$. If $bd=1$, then we get $a=c=0$ and $x_2^{d+1}-x_3^{b+1}=0$. This yields a contradiction since $I$ is prime.
\item[$\cdot$] If $(i,j)=(2,3)$, then we have
$(a,b)=(1,0)$ or $(0,1)$.
Thus we obtain $\deg x_1=\deg x_4=1$ or $x_2^2-x_3^2=0$ so this yields a contradiction.
\item[$\cdot$] If $(i,j)=(2,4)$, we have $\deg x_1=1$
and $(a,b)=(1,0)$ or $(0,1)$.
If $(a,b)=(1,0)$, we get $(c,d)=(0,1)$ and $x_2^2=x_1x_3$ so this contradicts $x_2 \in E_S$.
If $(a,b)=(0,1)$, we have $\deg x_i=1$ for all $1\leqq i \leqq 4$, this is a contradiction.
\item[$\cdot$] If $(i,j)=(3,4)$, by the same discussion as above, we get a contradiction.
\item[(ii)]
We can write
$X=\left[
\begin{array}{cc}
x_1 & -x_3 \\
-x_3^a & x_4 \\
x_2 & -x_1^b
\end{array}
\right]
$ for some $a,b\in \ZZ_{>0}$.
Moreover, we have
$x_{i},x_{j} \in E_S$
for some
$1 \leqq i<j \leqq 4$
by Proposition \ref{extremal}.
\item[$\cdot$] If $(i,j)=(1,2)$ or $(1,3)$ or $(1,4)$ or $(2,3)$ or $(3,4)$, then we have $x_1^{b+1}=x_2x_3$ or $x_3^{a+1}=x_1x_4$. This contradicts either $x_1 \in E_S$ or $x_3 \in E_S$.
\item[$\cdot$] If $(i,j)=(2,4)$, then $R$ is standard graded, this is a contradiction.
\item[(iii)]
We can write
$X=\left[
\begin{array}{cc}
x_1 & -x_3 \\
-x_2 & x_4 \\
x_3^ax_4^b & -x_1^cx_2^d
\end{array}
\right]
$ for some $a,b\in \ZZ_{>0}$.
Moreover, we have
$x_{i},x_{j} \in E_S$
for some
$1 \leqq i<j \leqq 4$
by Proposition \ref{extremal}.
\item[$\cdot$] If $(i,j)$ equals $(1,2)$, $(1,4)$, $(2,3)$ or $(3,4)$, then $R$ is standard graded, this is a contradiction.
\item[$\cdot$] If $(i,j)=(1,3)$, we get $R/(x_1,x_3)R \cong \kk$. Since $x_1,x_3 \in E_S$ is a regular sequence of $R$,
we get $R \cong A$. This is a contradiction.
\item[$\cdot$] If $(i,j)=(2,4)$, by the same discussion as above, we get a contradiction.
\end{itemize}

\item Let $d=3$.
Since $R$ is nearly Gorenstein,
We can write
$X=\left[
\begin{array}{cc}
x_1 & -x_4 \\
-x_2 & x_5 \\
x_3 & -x_1^ax_2^b
\end{array}
\right]$
for some $(a,b) \in \NN^2 \setminus \{(0,0)\}$
by Theorem \ref{trace}.
Moreover, we have
$x_{i},x_{j},x_{k} \in E_S$
for some
$1 \leqq i<j<k \leqq 5$
by Proposition \ref{extremal}.
\begin{itemize}
\item[$\cdot$] If $(i,j,k)=(1,2,3)$, then $R$ is standard graded, this is a contradiction.
\item[$\cdot$] If $(i,j,k)=(1,2,4)$, we have $x_1x_5=x_2x_4$
and $F=\NN \bfa_2+ \NN \bfa_4$ is a 2-dimensional face of $S$ where $x_i=\bfx^{\bfa_i} \in \kk[S]$ for $i=2,4$.
Thus we get $x_1 \in (x_2,x_4)R$. This yields a contradiction.
\item[$\cdot$] If $(i,j,k)=(1,2,5)$ or $(1,4,5)$ or $(2,4,5)$, by the same discussion as above, we get a contradiction.
\item[$\cdot$] If $(i,j,k)=(1,3,4)$ or $(1,3,5)$ or $(2,3,4)$ or $(2,3,5)$ or $(3,4,5)$, by the same argument as in $d=2$, it contradicts in this case as well.
\end{itemize}
\item Let $d=4$.
By the same argument as in $d=3$, it contradicts in this case as well.
\item Let $d\geqq 5$, $R$ cannot be nearly Gorenstein by Theorem \ref{trace}.
\end{itemize}
\end{proof}
\end{Thm}

\begin{Remark}
By the above proof, we also know that non-Gorenstein nearly Gorenstein semi-standard graded affine semigroup rings with projective dimension 2 do not exist when the Krull dimension is greater than or equal to 5.
\end{Remark}


\section{Almost Gorenstein semi-standard graded rings}
Let us recall the definition of the almost Gorenstein {\em graded} ring.
\begin{Definition}[{\cite[Definition 1.5]{goto2015almost}}]

We say that a Cohen--Macaulay graded ring $R$ is {\em almost Gorenstein} if there exists an exact sequence 
\begin{align}\label{ex_seq}
0 \rightarrow R \xrightarrow{\phi} \omega_R(-a) \rightarrow C \rightarrow 0
\end{align}
of graded $R$-modules with $\mu(C)=e(C)$, where $\phi$ is an injection of degree 0. 

\end{Definition}
From now, we will apply the discussion in \cite{higashitani2016almost} below to semi-standard graded rings.
First we consider the condition
\begin{align}\label{condition}
\text{there exists an injection $\phi : R \rightarrow \omega_R(-a)$ of degree 0}. 
\end{align}
This is a necessary condition for $R$ to be almost Gorenstein.
Let $C=\cok(\phi)$. Then $C$ is a Cohen--Macaulay $R$-module of dimension $d-1$ 
if $C\not=0$ (see \cite[Lemma 3.1]{goto2015almost}). 

The condition \eqref{condition} is satisfied if $R$ is a domain or generically Gorenstein and a level ring.
To prove this, we use the following well-known result.

\begin{Proposition}\label{nzd}
  Let $R$ be a semi-standard graded Cohen-Macaulay ring.
  If $R$ is a domain, or generically Gorenstein and a level ring,
  then there exists a homogeneous element $\omega_R$ of degree $-a(R)$ such that $R \cong Rx(-a(R))$.
\begin{proof}
While the proof for standard graded rings is given in \cite[Theorem 4.4.9]{bruns1998cohen}, it also works for semi-standard graded rings.
\end{proof}
\end{Proposition}

\begin{Proposition}\label{x}
When $R$ is a domain, or generically Gorenstein and a level ring, $R$ always satisfies the condition \eqref{condition}.
\end{Proposition}
\begin{proof}
By Proposition \ref{nzd},
we can pick a homogeneous element
$x \in (\omega_R)_{-a}$ such that $R$-homomorphism $\phi : R \xrightarrow{x} \omega_R(-a)$
is an injection of degree $0$, as desired.
\end{proof}

\begin{Remark}
Let $\kk$ be a finite field and let $R$ be a semi-standard graded ring with $R_0=\kk$.
Assume that $R$ satisfies the condition \eqref{condition}.
Put $K=\kk(x)$, then
\begin{align}
0 \rightarrow R \otimes_\kk K \rightarrow \omega_R(-a) \otimes_\kk K \rightarrow C \otimes_\kk K \rightarrow 0 \nonumber
\end{align}
is also exact. Moreover, we have
$\omega_R(-a) \otimes_\kk K=\omega_{R\otimes_\kk K}(-a)$
(\cite[see Exercise 3.3.31]{bruns1998cohen}),
$\dim C = \dim C \otimes_\kk K$, $\depth C = \depth C \otimes_\kk K$ and the Hilbert series of $C$ and $C \otimes_\kk K$ are equal.
\end{Remark}

\begin{Thm}\label{AGtokucho}
Assume that $R$ satisfies \eqref{condition}
and let $(h_0,h_1,\ldots,h_s)$ be the $h$-vector of $R$. Then the following is true.
\begin{itemize}
\item[(1)] We have $\mu(C)=r(R)-1$.
\item[(2)] The Hilbert series of $C$ is
\begin{align}\label{ccc}
\frac{\sum_{j=0}^{s-1}((h_s+\cdots+h_{s-j})-(h_0+\cdots+h_j))t^j}{(1-t)^{\dim R-1}}.
\end{align}
In particular, we have $e(C)=\sum_{j=0}^{s-1}((h_s+\cdots+h_{s-j})-(h_0+\cdots+h_j))$.
\end{itemize}
\begin{proof}

(1) follows from the same proof of \cite[Proposition 2.3]{higashitani2016almost}.

(2) We may assume $R$ is not Gorenstein.
Then there is the short exact sequence of graded $R$-module of degree $0$ as follows:
\begin{align}
0 \rightarrow R \rightarrow \omega_R(-a) \rightarrow C \rightarrow 0.
\end{align}
By Remark 3.4, we may assume $\kk$ is an infinite field.
Then we can show the statement in the same way as \cite[Theorem 2.1]{stanley1991hilbert}.
\end{proof}
\end{Thm}
From this Proposition, we have the Stanley's inequality (\cite[Theorem 2.1]{stanley1991hilbert}): 
\begin{Corollary}\label{prop1}
Assume that $R$ satisfies \eqref{condition}. Let $(h_0,h_1,\ldots,h_s)$ be the $h$-vector of $R$. Then we have the inequality 
$$h_s+\cdots+h_{s-j} \geqq h_0+\cdots+h_j$$ for each $j=0,1,\ldots,\lfloor s/2 \rfloor$.
\end{Corollary}

As the same proof of \cite[Corollary 2.7, Theorem 3.1 and Theorem 4.1]{higashitani2016almost}, we can prove the following.

\begin{Corollary}\label{tokuchou}
Assume that $R$ satisfies \eqref{condition}.
The following conditions are equivalent: 
\begin{itemize}
\item[(a)] there exists an injection $\phi : R \rightarrow \omega_R(-a)$ of degree 0 such that $C=\cok(\phi)$ satisfies $\mu(C)=e(C)$, 
namely, $R$ is almost Gorenstein; 
\item[(b)] every injection $\phi : R \rightarrow \omega_R(-a)$ of degree 0 satisfies $\mu(C)=e(C)$;
\item[(c)]$$
r(R)-1=\sum_{j=0}^{s-1}((h_s+\cdots+h_{s-j})-(h_0+\cdots+h_j)); 
$$
\end{itemize}
In particular, $\phi$ does not matter for the almost Gorenstein property of $R$. 
\end{Corollary}

\begin{Corollary}\label{suff}\label{Higashitani}
Let $R$ be a Cohen--Macaulay semi-standard graded ring with $h(R)=(h_0,h_1,\ldots,h_s)$ where $h_s \neq 0$.
Then the following is true.
\begin{itemize}
\item[(a)] If $R$ is domain and $h_i=h_{s-i}$ for $i=0,1,\cdots, \lfloor \frac{s}{2} \rfloor -1$, then $R$ is almost Gorenstein.
\item[(b)] If $R$ satisfies \eqref{condition} and
$s(R)=1$, then $R$ is always almost Gorenstein.
\end{itemize}
\end{Corollary}


Note that $R$ is generically Gorenstein if and only if $Q(R)$ is Gorenstein for Cohen-Macaulay ring $R$,
where $Q(R)$ is the total ring of fractions of $R$.
It is known that every almost Gorenstein ring is generically Gorenstein.

\begin{Lemma}[{\cite[Lemma 3.1(1)]{goto2015almost}}]\label{AGisGG}
Let $R \xrightarrow{\phi} \omega_R \rightarrow C \rightarrow 0$ be an exact sequence of $R$-modules.
If $\dim C \leqq d-1$, then $\phi$ is injective and $R$ is a generically Gorenstein ring.
In particular, if $R$ is almost Gorenstein, then $R$ is generically Gorenstein.
\end{Lemma}

\section{Almost Gorenstein property versus level property}
When $R$ is a standard graded ring, the next is known by \cite{goto2015almost}.
Actually, this result holds true even when $R$ is a semi-standard graded ring.
Here, we give another proof of \cite{goto2015almost} by using Stanley's enequality (Corollary \ref{prop1}).
\begin{Thm}\label{AGandLevel}
Let $R$ be a semi-standard Cohen--Macaulay graded ring with $\dim R>0$.
Suppose that $R$ is not Gorenstein.
Then the following conditions are equivalent:
\begin{itemize}
\item[(1)] $R$ is almost Gorenstein and level;
\item[(2)] $R$ is generically Gorenstein and $s(R)=1$.
\end{itemize}
\begin{proof}
First we show (2) implies (1).
Since $R$ is semi-standard graded ring,
we can check $s(R)=1$ implies $R$ is level.
Thus $R$ is almost Gorenstein by Proposition \ref{x} and Corollary \ref{Higashitani}(b).
We show (1) implies (2).
Since $R$ is generically Gorenstein by Lemma \ref{AGisGG},
it is enough to show $s(R)=1$. We assume $s(R) \geqq 2$.
Since $R$ is generically Gorenstein and level and almost Gorenstein,
we have $e(C)=r(A)-1=h_s-1$ by
Proposition \ref{x} and Corollary \ref{tokuchou}.
Therefore, we have
$$(s-1)(h_s-1)
=\sum_{j=1}^{\lfloor \frac{s}{2} \rfloor}(s-2j)(h_j-h_{s-j}).$$
Moreover,
we have the following enequalities
by Proposition \ref{x} and Corollary \ref{prop1}.
\begin{align}
h_s-1 &\geqq 0 \tag{*}\\
h_s-1 &\geqq (h_1-h_{s-1}) \tag{*1*}\\
h_s-1 &\geqq (h_1-h_{s-1}) + (h_2-h_{s-2}) \tag{*2*}\\
\vdots \notag \\
h_s-1 &\geqq (h_1-h_{s-1}) + (h_2-h_{s-2}) + \cdots + (h_{\lfloor \frac{s}{2} \rfloor}-h_{s-\lfloor \frac{s}{2} \rfloor}) \tag{*$\left \lfloor \frac{s}{2} \right \rfloor$*}
\end{align}
By $2\times \left({(\text{*}1\text{*})}+{(\text{*}2\text{*})}+\cdots+(\text{*}\left \lfloor \frac{s}{2} \right \rfloor-1\text{*}) \right)+\left(s-2 \lfloor \frac{s}{2} \rfloor\right) \times \left(\text{*}\lfloor \frac{s}{2} \rfloor\text{*}\right)$, we have
\begin{align}
(s-2)(h_s-1) \geqq \sum_{j=1}^{\lfloor \frac{s}{2} \rfloor}(s-2j)(h_j-h_{s-j})=(s-1)(h_s-1) \iff 0 \geqq h_s-1. \tag{**}
\end{align}
By (*) and (**),
we get $h_s=1$. So $R$ is Gorenstein. This yields a contradiction.
\end{proof}
\end{Thm}

Next, we will discuss non-level and almost Gorenstein semi-standard graded domains $R$ with small socle degree.
For $s(R)=2$,
the following result is known.
\begin{Corollary}[{see \cite[Corollary 3.11]{yanagawa1995castelnuovo} and \cite[Corollary 4.3]{higashitani2016almost}}]\label{s2ikahalevel}
Let $R$ be a standard graded domain.
If $s(R)\leqq 2$, then $R$ is level.
\end{Corollary}

In the case of semi-standard graded domain, the condition $s(R)=2$ does not implies level property in general.
The following is known about the Cohen-Macaulay type.

\begin{Proposition}[{\cite[Proposition 3.6]{higashitani2018non}}]\label{type}
Let $R$ be a Cohen--Macaulay semi-standard graded ring with the $h$-vector $(h_0,h_1,h_2)$. If $R$ is not level and $\kk[R_1]$ is a domain, then the Cohen-Macaulay type $r(R)$ of $R$ is equal to $h_1+h_2$.
\end{Proposition}

By using this, we have the following.

\begin{Proposition}
Let $R$ be a Cohen--Macaulay semi-standard graded domain with $s(R)=2$. The following conditions are equivalent:
\begin{itemize}
\item[(1)] $R$ is non-level and almost Gorenstein;
\item[(2)] $R$ is almost Gorenstein and $h(R)=(1,a,a+1)$ for some $a>0$;
\item[(3)] $R$ is non-level and $h(R)=(1,a,a+1)$ for some $a>0$;
\end{itemize}

\begin{proof}
First we show (1) implies (2).
Since $R$ is almost Gorenstein,
then we have
$2(h_2-1)=(h_1+h_2)-1$
by Corollary \ref{tokuchou} and
\ref{type}.
Thus $h(R)=(1,h_1,h_1+1)$ and $h_1>0$ since $R$ is non-level.
Next we show (2) implies (3).
If $R$ is level, then
we have $a=0$ since $R$ is almost Gorenstein.
Lastly, we show (3) implies (1).
Since $R$ is non-level,
we get $r(R)=2a+1$ by Proposition \ref{type}.
Then $R$ is almost Gorenstein by Proposition \ref{tokuchou}.
\end{proof}
\end{Proposition}

\begin{Remark}
Even if $R$ satisfies the condition $h(R)=(1,a,a+1)$ for some $a>0$,
$R$ is not necessarily non-level and almost Gorenstein.
Indeed, consider semi-standard graded ring $R=\kk[s,st,st^2,st^6,s^2t^5]$
with $\deg s=\deg st=\deg st^2=\deg st^6=1$ and $\deg s^2t^5=2$.
Then $h(R)=(1,2,3)$ but $R$ is level and non-almost Gorenstein.
\end{Remark}

\section{Almost Gorenstein property versus nearly Gorenstein property}
Lastly, we discuss the relation between almost Gorenstein property and nearly Gorenstein property.
The following theorem follows from \cite[Theorem 4.4]{miyashita2022levelness} and \cite[Theorem 4.7]{higashitani2016almost}.
Note that \cite[Theorem 4.4]{miyashita2022levelness} is standard graded version of Corollary \ref{NGhvector}.
\begin{Thm}\label{AGandNG!}
Let $R$ be a standard graded Cohen-Macaulay affine semigroup ring with $s(R) \geqq 2$. The following conditions are equivalent:
\begin{itemize}
\item[(1)] $R$ is almost Gorenstein and nearly Gorenstein;
\item[(2)] $R$ is Gorenstein.
\end{itemize}
\end{Thm}
Moreover,
it is known that every $1$-dimensional almost Gorenstein ring is nearly Gorenstein(see \cite[Proposition 6.1]{herzog2019trace}).
Therefore, we consider the comparison of nearly Gorenstein and almost Gorenstein properties in semi-standard graded affine semigroup rings when the socle degree and dimension are small.
In this Section, we show the following.
\begin{Thm}\label{AGisNGw}
Let $R$ be a non-standard semi-standard graded Cohen-Macaulay affine semigroup ring with $\dim R=s(R)=2$.
If $R$ is almost Gorenstein, then it is nearly Gorenstein.
\end{Thm}
To show this statement, we prove the following.
\begin{Thm}\label{nonlevelAGwithSocleDeg2}
Let $R=\kk[S]$ be a Cohen--Macaulay semi-standard graded affine semigroup ring with $\dim R=s(R)=2$. Then the following conditions are equivalent:
\begin{itemize}
\item[(1)] $R$ is non-level and almost Gorenstein;
\item[(2)] $S \cong \langle \{(2i,2n-2i): 0 \leqq i \leqq n \} \bigcup \{ (2j+2k-1,4n-2j-2k+1): 0 \leqq j \leqq n-1 \} \rangle$ for some $n \geqq 2$ and $1 \leqq k \leqq n+1$.
\end{itemize}
Moreover, if this is the case, then $R$ is always nearly Gorenstein and $h(R)=(1,n-1,n)$.
\begin{proof}
By using Lemma 3.3,
it is enough to show that (1) implies (2).
From the proof of \cite[Theorem 3.5]{higashitani2018non}, we get
$R \cong R_1 \oplus C$ as $R_1$-module.
Moreover, $B=R_1\cong \bigoplus_{i \in \NN} T_{ni}$ and $C(2) \cong \bigoplus_{i \in \NN} T_{n-1 + ni}$ for some $n \geqq 2$
where $T=\kk[x,y]$ and $n=h_1+1=h_2$.
Thus there exist
$V=\{ (id,(n-i)d): 0 \leqq i \leqq n \}$
such that $B=\kk[\langle V \rangle]$.
Note that $C$ has a minimal genearating system consisting of $n$ elements as a $B$-module, and all of its generators have degree 2.
Furthermore, since $R$ is finitely generated $B$-module,
the subset $W \subseteq S$ corresponding to the minimal generating system of $C$ satisfies $W \subseteq \RR_{\geqq 0}\langle V \rangle$
by Proposition \ref{extremal}.
Therefore, there exists
$W=\{(a+ik,2nd-a-ik) : 0 \leqq i \leqq n-1 \}$ such that
$C=B \langle{x_1}^{u_1}{x_2}^{u_2} : (u_1,u_2) \in W \rangle$
and $\kk[S] \cong \kk[\langle V \cup W\rangle]$
where $k>0$ and $0 < a < (n+1)d$ with $a \not\equiv 0\; (\textrm{mod}\; d$).
First we show
$d=k$.
Note that
$$
C_3 \supseteq \{ x_1^{a+d}x_2^{(3n-1)d-a}\} \cup \{x_1^{a+(i-1)k}x_2^{3nd-a-(i-1)k},\;
x_1^{a+(n-1)k+id}x_2^{(3n-i)d-a-(n-1)k} : 1 \leqq i \leqq n\}.
$$
If $d>k$, then
$2n=\dim_\kk T_{2n-1}=\dim_\kk C_3 \geqq 2n+1$, this yields a contradiction.
If $d \geqq k$, then we can check $x_1^{a+d}x_2^{(3n-1)d-a} \in \bigoplus_{2 \leqq i \leqq n}x_1^{a+(i-1)k}x_2^{3nd-a-(i-1)k}$.
Thus $d=jk$ for some $1 \leqq j \leqq n-2$.
Moreover,
since $$\left(x_1^{a+(n-2)k}x_2^{2nd-a-(n-2)d} \right)x_1^{nd} \in \bigoplus_{1 \leqq i \leqq n-1} x_1^{a+(n-1)k+id}x_2^{(3n-i)d-a-(n-1)k},$$
we get $k=(n-i)d$
for some $1 \leqq i \leqq n-1$.
Thus 
$d=k$.
Then we have
$$B_4 = \bigoplus_{0 \leqq i \leqq 4n} \kk x_1^{id}x_2^{(4n-i)d}, \;C_4 = \kk x_1^{a+(i-1)d}x_2^{(4n-i+1)d-a}\;\;\textit{and}\;\;R_4 = B_4 \oplus C_4.$$
For $x_1^{a}x_2^{2nd-a} \in R_2$,
we can check
$(x_1^{a}x_2^{2nd-a})^2 \in B_4$.
Note that $0 < a < (n+1)d$ and $a \not\equiv 0\; (\textrm{mod}\; d$).
Then there exists $1 \leqq l \leqq 2n+1$ such that $2a=ld$. Moreover, we can write $d=2m$ and $l=2k-1$ for some $m>0$ and $1 \leqq k \leqq n+1$.
On the other hand,
$\kk[S] \cong \kk[\langle V \cup W\rangle]$
implies $S \cong \langle V \cup W\rangle$ (see \cite[Theorem 2.1 (b)]{gubeladze1998isomorphism}).
Thus,
\begin{align}
\begin{split}\nonumber
S &\cong \langle \{(id,nd-id): 0 \leqq i \leqq n \} \cup \{ (a+jd,2nd-jd-a): 0 \leqq j \leqq n-1 \} \rangle\\
&\cong \langle \{(2i,2n-2i)): 0 \leqq i \leqq n \} \cup \{ (2j+l,4n-2j-2k+1): 0 \leqq j \leqq n-1 \} \rangle
\end{split}
\end{align}
for some $n \geqq 2$ and $1 \leqq k \leqq n+1$.
\end{proof}
\end{Thm}

\begin{proof}[Proof of Theorem \ref{AGisNGw}]
Since $R$ is not level by Theorem \ref{AGandLevel}, it is nearly Gorenstein by Theorem \ref{nonlevelAGwithSocleDeg2}.
\end{proof}

\begin{Examples}
For semi-standard graded affine semigroup rings where either the socle degree or the dimension is greater than 2, almost Gorenstein property does not imply nearly Gorenstein property in general. We can check the following is true in the same way as Examples \ref{exNG1}.
\begin{itemize}
\item[(1)] $R=\QQ[t,st,s^5t,s^4t^2]$ is non-nearly Gorenstein almost Gorenstein semi-standard graded affine semigroup ring with $\dim R=2$ and $s(R)=3$ where $\deg s^at^b=b$ for any $a,b \in \NN$.
From $h_3=1$, we can also confirm that $R$ is not nearly Gorenstein by Corollary \ref{NGhvector}.
\item[(2)]
$R=\QQ[u,s^2u,t^2u,t^4u,tu^2,t^3u^2]$
is non-nearly Gorenstein almost Gorenstein semi-standard graded affine semigroup ring with $\dim R=3$ and $h(R)=(1,1,2)$ where $\deg s^at^bu^c=c$ for any $a,b,c \in \NN$.
\item[(3)] $R=
\QQ[P_{1,1}]
\cong
\QQ[wyz^2,wy^2z,wx,wxyz,wxy^2z^2,w^2xy^2z^2,w^2xy^3z^3]$ $($see \cite[Theorem 4.5]{higashitani2018non}$)$ is non-nearly Gorenstein almost Gorenstein Ehrhart ring with $\dim R=4$ and $h(R)=(1,1,2)$.

\end{itemize}
\end{Examples}

\end{document}